\documentclass[12pt]{article}
\usepackage{latexsym, amsmath, amsfonts, amsthm, amssymb, a4wide}

\def \RR {\mathbb R}

\def \NN {\mathbb N}
\def \EE {\mathbb E}

\def \PP {\mathbb P}

\def \vphi {\varphi}

\def \cF {\mathcal F}

\def \cC { S }

\def \cI {\mathcal I}

\newtheorem{theorem}{Theorem}[section]
\newtheorem{lemma}[theorem]{Lemma}
\newtheorem{proposition}[theorem]{Proposition}
\newtheorem{corollary}[theorem]{Corollary}
\newtheorem{remark}[theorem]{Remark}

\newtheorem{definition}[theorem]{Definition}

\def\myffrac#1#2 in #3{\raise 2.6pt\hbox{$#3 #1$}\mkern-1.5mu\raise 0.8pt\hbox{$
#3/$}\mkern-1.1mu\lower 1.5pt\hbox{$#3 #2$}}

\def\qed{\hfill $\vcenter{\hrule height .3mm
\hbox {\vrule width .3mm height 2.1mm \kern 2mm \vrule width .3mm
height 2.1mm} \hrule height .3mm}$ \bigskip}

\title{An example related to the slicing inequality for general measures}
\author{Bo'az Klartag and Alexander Koldobsky}
\date{}

\begin{document}
\maketitle

\begin{abstract} For $n\in \NN,$ let $\cC_n$ be the smallest number $\cC>0$ satisfying the inequality
$$
\int_K f \le \cC \cdot |K|^{\frac 1n} \cdot \max_{\xi\in S^{n-1}} \int_{K\cap \xi^\bot} f
$$
for all centrally-symmetric convex bodies $K$ in $\RR^n$ and all even, continuous probability densities $f$ on $K$.
Here $|K|$ is the volume of $K$.
It was proved in \cite{Koldobsky-2014} that $\cC_n\le 2\sqrt{n}$, and in analogy with Bourgain's slicing problem,
it was asked whether $\cC_n$ is bounded from above by a universal constant.
In this note we construct an example showing that $\cC_n\ge c\sqrt{n}/\sqrt{\log \log n},$ where $c > 0$ is an absolute constant.
Additionally, for any $0 < \alpha < 2$ we describe a related example that satisfies the so-called $\psi_{\alpha}$-condition.
\end{abstract}

\section{Introduction}

Suppose that $K \subseteq \RR^n$ is a centrally-symmetric convex set of volume one (i.e., $K = -K$).
Given an arbitrary continuous probability density $f: K \rightarrow \RR$, can we find a hyperplane $H \subseteq \RR^n$ passing through
the origin such that
$$ \int_{H \cap K} f \geq c $$
where $c > 0$ is a universal constant, which is in particular independent of $K, f$ and even the dimension $n$?

\smallbreak
For many classes of convex bodies, the answer is surprisingly positive. It was proven by the second-named author \cite{Koldobsky-2015}
that the answer is affirmative in the case where $K \subseteq \RR^n$ is {\it unconditional}, i.e.,
$$ (x_1,\ldots,x_n) \in K \quad \Longleftrightarrow \quad (|x_1|,\ldots,|x_n|) \in K \qquad \qquad \text{for all} \ x = (x_1,\ldots,x_n) \in \RR^n. $$
This generalizes a result first proven by Bourgain \cite{Bou1}, who considered the case where the density $f$ is constant.
 Bourgain's investigations have led to the formulation of the slicing problem \cite{Bou1,Bou2}, which
asks whether $\sup_n L_n < \infty$, where $L_n > 0$ is the minimal number $L$ such that
for any centrally-symmetric convex body $K \subseteq \RR^n$,
$$ |K|_n \leq L \cdot \max_{\xi \in S^{n-1}} |K \cap \xi^{\perp}|_{n-1} \cdot |K|_n^{1/n}. $$
Here $\xi^\bot$ is the central hyperplane perpendicular to the vector $\xi \in S^{n-1}$, and $S^{n-1} = \{ x \in \RR^n \, ; \, |x| = 1 \}$
is the Euclidean unit sphere centered at the origin.
We write $|K|_{n}$ for the $n$-dimensional volume of $K$. When the dimension is clear from the context,
we will simply use $|K|$ in place of $|K|_n$.
Bourgain's slicing problem is still unsolved, the best-to-date estimate $L_n\le C n^{1/4}$ was established by the first-named author \cite{Klartag-2006},
removing a logarithmic term from an earlier estimate by Bourgain \cite{Bourgain-1991}. In analogy with the slicing problem,
for $n \geq 1$ let $\cC_n$ be the smallest number  $\cC > 0$ satisfying the inequality
\begin{equation}\label{hyperplane}
\mu(K) \le \cC \cdot \max_{\xi\in S^{n-1}} \mu^+(K\cap \xi^\bot) \cdot |K|^{\frac 1n}
\end{equation}
for all centrally-symmetric convex bodies $K \subseteq \RR^n$,
and all measures $\mu$ with a non-negative continuous density $f$ in $\RR^n$.
Here we abbreviate $$ \mu^+(K\cap\xi^\bot)=\int_{K\cap \xi^\bot} f $$ where the restriction of the density $f$ to $\xi^\bot$ is integrated
with respect to the Lebesgue measure in $\xi^\bot.$

\smallbreak
Many of the positive results
towards the slicing problem may be generalized from the case of the uniform measure on a convex domain $K$
to the broader class of any continuous probability density on $K$. Thus (\ref{hyperplane})
holds true, with $\cC$ having the order of magnitude of a universal constant,
whenever $K$ is the polar to a convex body with bounded volume ratio (see \cite{Koldobsky-2015}) or the unit
ball of a subspace of $L_p$ with $p>2$ (see \cite{KP}). The first result of this kind was proved in \cite{Koldobsky-2012}: If $K$ belongs to the class of intersection bodies $\cI_n$ (see definition in Section \ref{intbodies}), then (\ref{hyperplane})
holds with $\cC=2$ for all measures with even continuous densities.

\smallbreak In view of the positive results mentioned above,
one could think that perhaps $\sup_n \cC_n < \infty$. In this note
we show that this is not the case, and prove the following:

\begin{theorem} \label{sqrtn} There exist universal constants $c,C >0$ so that for any $n \geq 3$,
$$\frac{c\sqrt{n}}{\sqrt{\log \log n}}\le \cC_n \le C \sqrt{n}.$$
\label{thm_530}
\end{theorem}

The new result here is the left-hand side estimate. The right-hand side estimate was first established
in \cite{Koldobsky-2014}, and later a different
proof was found in \cite{CGL} where the central-symmetry assumption was no longer required.
In fact, the upper estimate for the constants $\cC_n$ may be deduced from the following theorem proved
in \cite[Corollary 1]{Koldobsky-2015}, which we now describe. A compact $K \subseteq \RR^n$ is star-shaped
if $t K \subseteq K$ for $0 \leq t \leq 1$,
where $t K = \{ t x \,; \, x \in K  \}$. We say that a star-shaped $K$ is a star body if its radial function
$$\rho_K(x) = \max\{a \geq 0:\ ax\in K\} \qquad \qquad (x \in S^{n-1}) $$
is continuous and positive in $S^{n-1}$. For a star body $K \subseteq \RR^n$ denote by
$$d_{\rm {ovr}}(K, \cI_n) = \inf \left\{ \left( \frac {|D|}{|K|}\right)^{1/n};\ K\subseteq D,\ D\in \cI_n \right\}$$
the outer volume ratio distance from $K$ to the class of intersection bodies $\cI_n.$
\medbreak
\begin{theorem} \label{ovr} For any $n \in \NN,$ any centrally-symmetric
star body $K \subseteq \RR^n,$ and any measure $\mu$
with a continuous density on $K$,
$$\mu(K)\le 2\ d_{\rm ovr}(K,\cI_n) \cdot \max_{\xi \in S^{n-1}} \mu^+(K\cap \xi^\bot) \cdot |K|^{1/n}.$$
\end{theorem}
The right-hand estimate of Theorem \ref{sqrtn} follows from Theorem \ref{ovr} and John's theorem,
since all elllipsoids are intersection bodies (see \cite{Koldobsky-2014}).
For the sake of completeness, we present a short proof of Theorem \ref{ovr} and related results
in Section \ref{intbodies}. In Section 3 we move on to discuss the lower estimate for $\cC_n$ which shows that
the $\sqrt{n}$ upper bound is in fact optimal up to a $\log \log$-term:

\begin{theorem} For any $n \geq 3$ there exists a centrally-symmetric convex body $T \subseteq \RR^n$ and
 an even,  continuous probability density $f: T \rightarrow [0, \infty)$ such that for any affine hyperplane $H \subseteq \RR^n$,
  \begin{equation}  \int_{T \cap H} f \leq C \frac{\sqrt{\log \log n}}{\sqrt{n}} \cdot |T|^{-1/n}, \label{eq_649}
 \end{equation}
where $C > 0$ is a universal constant.
\label{thm_639} \end{theorem}
Note that the hyperplane $H$ in Theorem \ref{thm_639} is not required to pass through the origin.
The combination of Theorem \ref{ovr} and Theorem \ref{thm_639} implies the following:

\begin{corollary}
	There exists a centrally-symmetric convex body $T \subseteq \RR^n$ with $d_{\rm ovr}(T, \cI_n) \geq c \sqrt{n} / \sqrt{\log \log n}$, where $c > 0$ is a universal constant.
\end{corollary}

For $\alpha \in (0,2]$ we say that a measure $\mu$ on $\RR^n$ admits $\psi_{\alpha}$-tails with parameters $(\beta, \gamma)$ if for any linear functional $\ell: \RR^n \rightarrow \RR$
\begin{equation}
\mu \left( \left \{ x \in \RR^n \, ; \, |\ell(x)| \geq t E \right \} \right) \leq \beta \exp(-\gamma t^{\alpha}) \cdot \mu(\RR^n) \qquad \qquad (\textrm{for all} \ t > 0)
\label{eq_935}
\end{equation}
where $E = \int_{\RR^n} |\ell(x)| d \mu(x)$.
It follows from the Brunn-Minkowski inequality that the uniform probability measure $\mu$ on a convex body in $\RR^n$ has  $\psi_1$-tails with parameters $(\beta, \gamma)$
that are universal constants, see, e.g., \cite[Section 2.4]{BGVV}. %In short, we say that such $\mu$ has uniformly subexponential tails.
It follows from the argument by Bourgain (see e.g. \cite[Section 3.3]{BGVV}) that for any  measure
$\mu$ with an even, continuous density supported on a centrally-symmetric convex body $K \subseteq \RR^n$,
\begin{equation}  \mu(K) \leq C(\beta,\gamma) \cdot n^{(2 - \alpha)/4} \log n \cdot \sup_{H \subseteq \RR^n} \mu^+(K \cap H) \cdot |K|^{1/n},
\label{eq_1009} \end{equation}
where the supremum runs over all $(n-1)$-dimensional affine hyperplanes $H \subseteq \RR^n$, where $\mu$ is assumed to have $\psi_{\alpha}$-tails
with parameters $(\beta, \gamma)$, and where $C(\beta, \gamma) > 0$
depends only on $\beta$ and $\gamma$. For completeness, we provide a short argument explaining (\ref{eq_1009}) in an appendix.

\smallskip
Specializing (\ref{eq_1009}) to the case $\alpha = 1$, we obtain the bound $L_n \leq C n^{1/4} \log n$ for Bourgain's slicing problem, which is
not far from the best estimate known to date. In the log-concave case it was proven in \cite{KM} that  the logarithmic factor in (\ref{eq_1009}) is not needed.
The following theorem establishes the near-optimality of the bound (\ref{eq_1009}), up to logarithmic terms:

\begin{theorem} For any $n$ and $0 < \alpha \leq 2$ there exists a centrally-symmetric convex body $T \subseteq \RR^n$ and
 an even,  continuous probability density $f: T \rightarrow [0, \infty)$ with the following properties:
 \begin{enumerate}
 \item[(i)] $\displaystyle \int_{T \cap H} f \leq C_{\alpha} \cdot  n^{(\alpha - 2)/4} \cdot |T|^{-1/n}$ for any affine hyperplane $H \subseteq \RR^n$.
 \item[(ii)] The measure whose density is $f$ admits $\psi_{\alpha}$-tails with parameters $(\tilde{c}_{\alpha}, \tilde{C}_{\alpha})$.
 \end{enumerate}
 Here, $\tilde{c}_{\alpha}, C_{\alpha}, \tilde{C}_{\alpha} > 0$ depend solely on $\alpha \in (0, 2]$.
 \label{thm_1026}
\end{theorem}

Theorem \ref{thm_1026} shows that Bourgain's slicing problem cannot be resolved on the affirmative
if all that is used is the uniformly subexponential tails of linear functionals on convex bodies.

\medskip
Throughout this paper, unless specified otherwise we write $c, C, \tilde{C}$ etc. for various positive, universal constants, whose value is not necessarily
the same in different appearances. We use lower-case $c, \tilde{c}, c_1$ for sufficiently small positive universal constants,
while $C, C_1,\tilde{C}_1$ etc. are sufficiently large universal constants. A convex body $K$ in $\RR^n$
is a compact, convex set with a non-empty interior.
 The standard scalar product between $x, y \in \RR^n$ is denoted by $x \cdot y$ or by $\langle x, y \rangle$.
We write $\log$ for the natural logarithm.

\bigbreak
{\bf Acknowledgements.} We would like to thank Sergey Bobkov for encouraging us to discuss and
think on this problem. We thank the anonymous referee for a thoughtful report. A major part of the work was done when both authors were visiting
the Banff International Research Station during May 21--26, 2017. We would like to thank BIRS for hospitality.
The first-named author was supported in part by a European Research Council (ERC) grant.
The second-named author was supported in part by the US National Science Foundation
grant DMS-1700036.

\section{The outer volume ratio distance from intersection bodies} \label{intbodies}

The {\it Minkowski functional} of a star body $D \subseteq \RR^n$ is defined by
$$\|x\|_D=\min\{a\ge 0:\ x\in aK\} \qquad \qquad (x \in S^{n-1}). $$
Note that $\| x \|_D^{-1} = \rho_D(x)$ for any $x \in S^{n-1}$, where $\rho_D$ is the radial
function of $D$.
The class of intersection bodies was introduced by Lutwak \cite{Lutwak}.
Let $D, L$ be origin-symmetric star bodies in $\RR^n.$ We say that $D$ is the
intersection body of $L$ if the radius of $K$ in every direction is
equal to the $(n-1)$-dimensional volume of the section of $L$ by the central
hyperplane orthogonal to this direction, i.e. for every $\xi\in S^{n-1},$
$$
\rho_D(\xi) = |L\cap \xi^\bot|
= \frac 1{n-1} \int_{S^{n-1}\cap \xi^\bot} \rho_L^{n-1}d\theta=
\frac 1{n-1} R\left(\rho_L^{n-1}\right)(\xi),$$
where $R:C(S^{n-1})\to C(S^{n-1})$ is the {\it spherical Radon transform}
$$Rg(\xi)=\int_{S^{n-1}\cap \xi^\bot} g(x) dx,\qquad \forall g\in C(S^{n-1}).$$
All star bodies $K$ that appear as intersection bodies of star bodies
form {\it the class of intersection bodies of star bodies}.

\smallbreak A more general class of {\it intersection bodies}
is defined as follows; see \cite{GLW}. If $\nu$ is a finite Borel measure on $S^{n-1},$ then
the spherical Radon transform $R\nu$ of $\nu$ is a functional on $C(S^{n-1})$ acting by
$$(R\nu, g)=(\nu, Rg)=\int_{S^{n-1}} Rg(x) d\nu(x),\qquad \forall g\in C(S^{n-1}).$$

\begin{definition}A star body $D$ in $\RR^n$ is called an {\it intersection body}, and we write $D\in \cI_n,$
if there exists a finite Borel measure
$\nu_D$ on $S^{n-1}$ such that
$\rho_D=R\nu_D$ as functionals on $C(S^{n-1}),$  i.e.
\begin{equation}\label{intbody}
\int_{S^{n-1}} \rho_D(x) g(x) dx = \int_{S^{n-1}} Rg(x)d\nu_D(x),\qquad \forall g\in C(S^{n-1}).
\end{equation}
\end{definition}

For example, let us consider the cross-polytope
$$B_1^n= \left \{x\in \RR^n \, ; \, \|x\|_1=\sum_{k=1}^n |x_k| \le 1 \right \}.$$
It was proved in \cite{Koldobsky-1998} that $B_1^n$ is an intersection body. To see this, note that
the function $e^{-\|\cdot\|_1}$ is the Fourier transform of the function
$$\phi(\xi)=\frac 1{\pi^n}\prod_{k=1}^n \frac 1{1+\xi_k^2},\qquad \qquad (\xi\in \RR^n),$$
and use the connection between the Radon and Fourier transforms: For $x \in S^{n-1}$,
\begin{align*} \rho_{B_1^n}(x) & = \|x\|_1^{-1} = \frac 12\int_{\RR} e^{-t\|x\|_1} dt =
\frac 1{\pi^{n-1}} \int_{x^\bot} \prod_{k=1}^n \frac 1{1+\xi_k^2}d\xi \\ &
=\frac 1{\pi^{n-1}}\int_{S^{n-1}\cap x^\bot} \left(\int_0^\infty t^{n-2}\prod_{k=1}^n \frac 1{1+t^2\xi_k^2} dt\right) d\xi. \end{align*}
We get that the radial function of the cross-polytope is the spherical Radon transform of the function
$$\xi\to \frac 1{\pi^{n-1}} \int_0^\infty t^{n-2}\prod_{k=1}^n \frac 1{1+t^2\xi_k^2} dt.$$
This function is integrable on the sphere, but it is not bounded (it takes infinite values on a set of measure zero).
Therefore, $B_1^n$ is an intersection body, but not the intersection body of a star body; see \cite{Koldobsky-1998}
or \cite[Section 4.3]{Koldobsky-book} for details. Note that it was proved in \cite{Koldobsky-1998} that all
polar projection bodies are intersection bodies.

\smallbreak
It was proven  in \cite{Koldobsky-2015} that $d_{\rm ovr}(K,\cI_n)\le e$ for every
unconditional convex body $K$ in $\RR^n.$ In fact, by a result of Lozanovskii \cite{Lo}
(see the proof in \cite[Corollary 3.4]{pisier}), there exists a linear
operator $T$ on $\RR^n$ so that
$T(B_\infty^n) \subset K \subset n T(B_1^n),$
where $B_\infty^n$ is the cube with sidelength 2 in $\RR^n.$
Let $D=nT(B_1^n).$ From the fact that a linear transformation of an intersection body
is an intersection body, the body $D$ is an intersection body in $\RR^n.$
Since $|B_1^n| = 2^n/n!$, we have
$|D|^{1/n}\le 2e |\det T|^{1/n}.$ On the other hand, $|T(B_\infty^n)| = 2^n |\det T|,$ and $T(B_\infty^n)\subset K,$
so $|D|^{1/n} \le e\  |K|^{1/n}.$

\medbreak
We now present a proof of Theorem \ref{ovr} that is slightly shorter than that in \cite{Koldobsky-2015}.

\begin{proof}[{\bf Proof of Theorem \ref{ovr}}] For every $\xi\in {S^{n-1}},$ we have
\begin{equation}\label{max}
\mu^+(K\cap \xi^\bot) \le \max_{\theta\in S^{n-1}} \mu^+ (K\cap \theta^\bot).
\end{equation}
Let $f$ be the continuous density of the measure $\mu.$ Writing the integral in spherical coordinates, we see
that for every $\xi\in S^{n-1}$,
$$\mu^+(K\cap \xi^\bot)= \int_{K\cap \xi^\bot} f =
\int_{S^{n-1}\cap \xi^\bot} \left(\int_0^{\rho_K(\theta)} r^{n-2}f(r\theta)\ dr \right)d\theta
= \int_{S^{n-1}\cap \xi^\bot} F(\theta) d \theta, $$
where
$$ F(\theta) = \int_0^{\rho_K(\theta)} r^{n-2}f(r \theta)\ dr \qquad \qquad (\theta \in S^{n-1}). $$
Therefore, inequality (\ref{max}) can be written in terms of the spherical Radon transform
\begin{equation}\label{max-radon}
R F (\xi)\le \max_{\theta\in S^{n-1}} \mu^+(K\cap \theta^\bot)
\end{equation}
for all $\xi\in S^{n-1}.$ Note that the right-hand side of (\ref{max-radon})
does not depend on $\xi$.

\smallbreak Let $D$ be an intersection body such that the distance $d_{\rm {ovr}}(K,\cI_n)$ is almost realized, i.e.
$K\subset D$ and for some small $\delta > 0$,
\begin{equation}\label{eq2}
|D|^{1/n}\le (1+\delta)d_{\rm {ovr}}(K,\cI_n)|K|^{1/n}.
\end{equation}
Integrating both sides of (\ref{max-radon}) by $\xi$ over the sphere with respect to the measure $\nu_D$
corresponding to $D$ by definition (\ref{intbody}), we get
\begin{equation}\label{eq3}
\int_{S^{n-1}} \rho_D(\theta) \left(\int_0^{\rho_K(\theta)} r^{n-2} f(r\theta) dr\right)d\theta\le
\nu_D(S^{n-1}) \max_{\theta\in S^{n-1}} \mu^+(K\cap \theta^\bot).
\end{equation}
The left-hand side of (\ref{eq3}) is equal to
$$\int_{S^{n-1}}  \left(\int_0^{\rho_K(\theta)}(\rho_D(\theta)-r) r^{n-2} f(r\theta) dr\right)d\theta\
+\ \int_{S^{n-1}}  \left(\int_0^{\rho_K(\theta)} r^{n-1} f(r\theta) dr\right)d\theta$$
\begin{equation}\label{left}
\ge \int_{S^{n-1}}  \left(\int_0^{\rho_K(\theta)} r^{n-1} f(r\theta) dr\right)d\theta = \int_K f= \mu(K),
\end{equation}
because $K\subset D$ implies $\rho_D(\theta)\ge \rho_K(\theta)$ for every $\theta.$

\smallbreak
Now we estimate the left-hand side of (\ref{eq3}) from above. We use $R1(\xi) = |S^{n-2}|$ for every $\xi\in S^{n-1},$  definition (\ref{intbody}), H\"older's inequality
and a standard formula for volume:
$$\nu_D(S^{n-1}) =\frac 1{|S^{n-2}|}\int_{S^{n-1}} R1(\xi) d\nu_D(\xi) = \frac 1{|S^{n-2}|}\int_{S^{n-1}}\rho_D(\xi) d\xi$$
$$\le \frac {|S^{n-1}|^{\frac{n-1}n}}{|S^{n-2}|}\left(\int_{S^{n-1}}\rho_D^n(\xi) d\xi\right)^{\frac 1n}\le 2|D|^{\frac 1n}.$$
By using (\ref{eq2}), sending $\delta\to 0,$ and combining
the estimates above, we obtain  the conclusion of the theorem.
Note that the uniform measure on the sphere was not normalized
in the calculations.
\end{proof}

\section{A counterexample}

We move on to the proof of Theorem \ref{thm_639}. We may clearly assume that the dimension $n$ exceeds a given
universal constant $C$, as otherwise the conclusion of the theorem is trivial.
We shall  need the following well-known Bernstein-type inequality. A proof is provided for completeness:

\begin{lemma} Let $Y_1,\ldots,Y_N$ be independent, identically-distributed random variables attaining values in the interval $[0,1]$.
Let $p \in [0,1]$ satisfy $p \geq \EE Y_1$. Then,
$$ \PP \left( \frac{1}{N} \sum_{i=1}^N Y_i \geq 3 p \right) \leq e^{-p N}.  $$
\label{lem_512}
\end{lemma}

\begin{proof} Since $Y_1 \in [0,1]$ with $\EE Y_1 \leq p$,
$$ \EE e^{Y_1} =  1 + \sum_{q=1}^{\infty} \frac{\EE Y_1^{q-1} Y_1}{q!}
\leq 1 + \sum_{q=1}^{\infty} \frac{\EE Y_1}{q!} \leq 1 + p (e-1). $$
Therefore,
$$ \EE e^{\sum_{i=1}^N Y_i} = \left( \EE e^{Y_1} \right)^N
\leq  \left( 1 + p (e-1) \right)^N \leq e^{ N p (e-1) }.  $$
By the Markov-Chebyshev inequality,
\begin{equation*}  \PP \left( \frac{1}{N} \sum_{i=1}^N Y_i \geq 3 p \right) \leq e^{-3 N p} \EE e^{\sum_{i=1}^N Y_i} \leq e^{N p (e - 4)} < e^{-p N}.
\tag*{\qedhere}
\end{equation*}
\end{proof}

For $t \in \RR$ we set $\vphi(t) = e^{-t^2/2}$.
Later on, we will apply Lemma \ref{lem_512} for $Y_i = \vphi(t + R \Theta_i \cdot \xi)$, where $\Theta_i$ is
a random point in the sphere $S^{n-1}$ and $\xi$ is a fixed unit vector.

\begin{lemma} Let $n \geq 4$ and let $\Theta \in S^{n-1}$ be a random point, distributed uniformly over $S^{n-1}$. Let $\xi \in S^{n-1}$ be a fixed unit vector.
Then for any $t \in \RR$ and $R > \sqrt{n}$,
$$ \EE \vphi \left(t + R \Theta \cdot \xi \right) \leq C \frac{\sqrt{n}}{R} \cdot \vphi \left( \frac{c \sqrt{n}}{R} t \right), $$
where $C > 0$ and $0 < c < 1$ are universal constants.
 \label{lem_540}
\end{lemma}

\begin{proof} Denote $Z = \Theta \cdot \xi$. Then $Z$ is a random variable supported in the interval $[-1, 1]$ whose density in this
interval is proportional to the function $s \mapsto (1 - s^2)^{(n-3)/2}$. Setting $k = n-3$ we see that we need to prove that
\begin{equation}  \alpha_k \int_{-1}^1 \vphi(R s + t) \cdot \left( 1 - s^2 \right)^{k/2} ds \leq C \frac{\sqrt{k}}{R} \cdot e^{-\frac{k t^2}{2 R^2}}, \label{eq_533}
\end{equation}
where
\begin{equation}
 \alpha_k^{-1} = \int_{-1}^1 \left( 1 - s^2 \right)^{k/2} ds
\geq \int_{-1/\sqrt{k}}^{1/\sqrt{k}} \left( 1 - s^2 \right)^{k/2} ds \geq \frac{c}{\sqrt{k}}. \label{eq_333}
 \end{equation}
In order to prove (\ref{eq_533}), we note that
$$ \alpha_k \int_{-1}^1 \vphi(R s + t) \cdot \left( 1 - s^2 \right)^{k/2} ds \leq C \sqrt{k} \int_{-\infty}^{\infty} e^{-(R s + t)^2/2 - k s^2/2} ds
= C \sqrt{\frac{2 \pi k}{R^2 + k}} \cdot e^{-\frac{k t^2}{ 2 (R^2 + k)}},
$$
 where we used (\ref{eq_333}) and the inequality $(1 - \alpha)^m \leq \exp(-\alpha m)$, valid for all $0 < \alpha < 1$ and $m > 0$.
 Thus (\ref{eq_533}) is proven.
 \end{proof}

By combining Lemma \ref{lem_512} and Lemma \ref{lem_540} we obtain the following:

\begin{corollary} Let $N \geq n \geq 4$ and let $\Theta_1,\ldots, \Theta_N \in S^{n-1}$ be independent, identically-distributed
random vectors, distributed uniformly over the sphere $S^{n-1}$. Fix $\xi \in S^{n-1}, t \in \RR, R > \sqrt{n}$ and $\alpha > 0$. Then,
$$ \PP \left( \, \frac{1}{N} \sum_{i=1}^N \vphi \left( t + R \xi \cdot \Theta_i   \right) \geq C \max \left \{ \alpha,
\frac{\sqrt{n}}{R} \cdot \vphi \left(\frac{c \sqrt{n}}{{R}} t  \right) \right \} \, \right) \leq e^{-N \alpha}, $$ \label{cor_554}
where $C > 0$ and $0 < c < 1$ are universal constants.
\end{corollary}

\begin{proof} Set $Y_i = \vphi( t + R  \xi \cdot \Theta_i )$. Then $Y_1,\ldots, Y_N$ are independent, identically-distributed random variables
attaining values in the interval $[0,1]$. Set $$ p =  \max \left \{ \alpha,
C\frac{\sqrt{n}}{R} \cdot \vphi \left(\frac{c \sqrt{n}}{{R}} t  \right) \right \}, $$
where $c,C > 0$ are the constants from Lemma \ref{lem_540}. Then $\EE Y_1 \leq p$, according to Lemma \ref{lem_540}. By Lemma \ref{lem_512},
$$ \PP \left( \frac{1}{N} \sum_{i=1}^N Y_i \geq 3  p  \right) \leq \exp(-N p)  \leq \exp(-N \alpha), $$
where we used that $p \geq \alpha$ in the last passage.
\end{proof}

The function $\vphi(s) = e^{-s^2/2}$ has a bounded derivative $\vphi^{\prime}(s) = -s e^{-s^2/2}$. Therefore $\vphi$ is a $1$-Lipschitz function
on the entire real line. This Lipschitz property enables us to make the estimate of Corollary \ref{cor_554} uniform in $\xi \in S^{n-1}$,
as explained in the following:

\begin{proposition} Assume that $n \geq 5$, that $N \geq 10 n \log n$ and that $\sqrt{n} \leq R \leq n$.
Let $\Theta_1,\ldots, \Theta_N$ be independent, identically-distributed
random vectors,  distributed uniformly in $S^{n-1}$.
Then with probability of at least $1 - n^{-n}$
the following holds: For all $\xi \in S^{n-1}$ and $t \in \RR$,
\begin{equation}
 \frac{1}{N} \sum_{i=1}^N \vphi \left( t + R \xi \cdot \Theta_i   \right) \leq \frac{C n \log n}{\min \{N, n^3\}} \, + \, C
\frac{\sqrt{n}}{R} \cdot \vphi \left(\frac{c \sqrt{n}}{{R}} t  \right),
\label{eq_608} \end{equation} \label{prop_612}
where $c, C > 0$ are universal constants.
\end{proposition}

\begin{proof} For all possible choices of $\theta_1,\ldots, \theta_N \in S^{n-1}$, the function
$$ G_{\theta_1,\ldots,\theta_N}(t, \xi) = \frac{1}{N} \sum_{i=1}^N \vphi( t + R \xi \cdot \theta_i ) \qquad \qquad (t \in \RR, \xi \in S^{n-1}) $$
is a Lipschitz function on $\RR \times S^{n-1}$ whose Lipschitz constant is at most $R + 1 \leq n+1$. Set $\delta = n^{-3}$, and let $\cF \subseteq S^{n-1}$
be a $\delta$-net, i.e., for any $x \in S^{n-1}$ there exists $y \in \cF$ with $|x-y| \leq \delta$. By a standard volumetric argument
(see, e.g., \cite{pisier}), there exists a $\delta$-net $\cF \subseteq S^{n-1}$ with cardinality
\begin{equation}  \#(\cF) \leq \left( \frac{5}{\delta} \right)^n \leq e^{6 n \log n}. \label{eq_602} \end{equation}
Let $I \subseteq \RR$ be the set of all integer multiples of $n^{-3}$ that lie in the interval $[-n^3, n^3]$.
Then for any $\xi \in S^{n-1}$ and $t \in [-n^3, n^3]$ there exists $\tilde{\xi} \in \cF$ and $\tilde{t} \in I$ with
\begin{equation}
G_{\theta_1,\ldots,\theta_N}(t, \xi) \leq (n+1) \cdot \frac{2}{n^3} + G_{\theta_1,\ldots,\theta_N}(\tilde{t}, \tilde{\xi}),
\label{eq_252}
\end{equation}
by the aforementioned  Lipschitz property
of $G_{\theta_1,\ldots,\theta_N}$. Let us now apply  Corollary \ref{cor_554} with $\alpha = 10 (n \log n) / N$, to obtain
\begin{align} \label{eq_221} \PP & \left( \forall \xi \in \cF, t \in I, \quad  G_{\Theta_1,\ldots,\Theta_N}(t,\xi) \leq C \max \left \{ \alpha,
\frac{\sqrt{n}}{R} \cdot \vphi \left(\frac{c \sqrt{n}}{{R}} t  \right) \right \} \right) \\ & \geq 1 - \#(\cF) \cdot \#(I) \cdot e^{-N \alpha} \geq 1 - n^7 \cdot e^{6 n \log n} \cdot e^{-10 n \log n} \geq 1 - n^{-n}. \nonumber \end{align}
The constant $c$ in (\ref{eq_221}) is at most one. Hence for any $t, \tilde{t} \in [-n^3, n^3]$ with $|t - \tilde{t}| \leq n^{-3}$ we have that
$\vphi(c \sqrt{n} \cdot \tilde{t} / R) \leq 5 \vphi(c \sqrt{n} \cdot t / R)$. We therefore conclude from (\ref{eq_252}) and (\ref{eq_221}) that
\begin{align*}  \PP  \left( \forall \xi \in S^{n-1}, |t| \leq n^3, \quad  G_{\Theta_1,\ldots,\Theta_N}(t,\xi) \leq
\frac{C n \log n}{N } +
\frac{\tilde{C} \sqrt{n}}{R} \cdot \vphi \left(\frac{c \sqrt{n}}{{R}} t  \right) \right)  \geq 1 - n^{-n}. \nonumber \end{align*}
We have thus proven that with probability of at least $1-n^{-n}$, inequality (\ref{eq_608}) holds true
for all $\xi \in S^{n-1}$ and $|t| \leq n^3$. The validity of (\ref{eq_608}) when $|t| > n^3$ is much easier, as in this case
$$ \frac{1}{N} \sum_{i=1}^N \vphi \left( t + R \xi \cdot \Theta_i   \right) \leq \frac{1}{N} \sum_{i=1}^N \vphi \left( |t| - n \right) \leq \vphi(n^3 - n) \leq e^{-10 n} \leq \frac{n \log n}{\min \{N, n^3\}},  $$
with probability one. This completes the proof.
\end{proof}

\begin{remark}{\rm
The use of the $\delta$-net in the proof of Proposition \ref{prop_612}
is probably not optimal, and it is the reason for the  appearance
of the $\log \log$-factor in our result. We suspect that this factor may be improved
or eliminated by using more sophisticated tools from the theory of Gaussian processes,
such as Slepian's lemma, Talgrand's majorizing measure or related results. In fact, perhaps Gordon's minmax theorem
may be used in order to improve the double logarithm to a triple logarithm in Theorem \ref{thm_530}. These considerations
will be expanded upon elsewhere.
} \label{rem_1131}
\end{remark}

By iterating Proposition \ref{prop_612} twice we obtain the following:

\begin{lemma} Assume that $n \geq C_1$ and let us set
 \begin{equation} N_1 = n^3, \quad N_2 = \lceil n \log^3 n \rceil, \quad R_1 = n / \sqrt{\log n} \quad \text{and} \quad  R_2 = n / \sqrt{\log \log n}. \label{eq_130} \end{equation}
Then there exist unit vectors $\theta_1,\ldots,\theta_{N_1}, \eta_1,\ldots, \eta_{N_2} \in S^{n-1}$ such that for all $\xi \in S^{n-1}$ and $t \in \RR$,
$$
 \frac{1}{N_1 N_2} \sum_{i=1}^{N_1} \sum_{j=1}^{N_2} \vphi \left( t + R_1 \xi \cdot \theta_i  + R_2 \xi \cdot \eta_j \right) \leq \frac{C}{\sqrt{n} \log n} \, + \, C
\frac{\sqrt{n}}{R_2} \cdot \vphi \left(c \frac{\sqrt{n}}{{R_2}} t  \right), $$
where $c, C, C_1 > 0$ are universal constants.
\label{lem_955}
\end{lemma}

\begin{proof}
We may assume that $n \geq 5$ is sufficiently large so that $c R_2 / R_1 > 1$, where throughout this proof $c > 0$ is the constant from Proposition \ref{prop_612}.
By the conclusion of Proposition \ref{prop_612}, we may fix unit vectors $\theta_1,\ldots,\theta_{N_1} \in S^{n-1}$
such that for all $\xi \in S^{n-1}$ and $s \in \RR$,
\begin{equation}
 \frac{1}{N_1} \sum_{i=1}^{N_1} \vphi \left( s + R_1 \xi \cdot \theta_i   \right) \leq \frac{C n \log n }{N_1} \, + \, C
\frac{\sqrt{n}}{R_1} \cdot \vphi \left(\frac{c \sqrt{n}}{{R_1}} s  \right). \label{eq_1115}
\end{equation}
Note that $C n (\log n) / N_1 \leq C / n$. In particular, for any choice of $\eta_1,\ldots,\eta_{N_2} \in S^{n-1}$, and for any $\xi \in S^{n-1}$ and $t \in \RR$,
$$
 \frac{1}{N_1 N_2} \sum_{i=1}^{N_1} \sum_{j=1}^{N_2} \vphi \left( t + R_1 \xi \cdot \theta_i  + R_2 \xi \cdot \eta_j \right) \leq \frac{C}{n} \, + \, C
\frac{\sqrt{n}}{R_1} \cdot \frac{1}{N_2} \sum_{j=1}^{N_2} \vphi \left(\frac{c \sqrt{n}}{{R_1}} t  \, + \, \frac{c R_2 \sqrt{n}}{R_1} \xi \cdot \eta_j  \right), $$
where we used (\ref{eq_1115}) with $ s = t + R_2 \xi \cdot \eta_j$.
Let us now apply Proposition \ref{prop_612} with $$ R = c R_2 \sqrt{n} / R_1 > \sqrt{n} \qquad \text{and} \qquad  N = N_2 \geq n \log^3 n \geq 10 n \log n. $$
From the conclusion of this proposition we obtain that there exist
unit vectors $\eta_1,\ldots,\eta_{N_2} \in S^{n-1}$ such that for all $\xi \in S^{n-1}$ and $t \in \RR$,
\begin{align*}  \frac{1}{N_2} \sum_{j=1}^{N_2} \vphi \left(\frac{c \sqrt{n}}{{R_1}} t  \, + \, \frac{c R_2 \sqrt{n}}{R_1} \xi \cdot \eta_j  \right)
& \leq \frac{C n \log n }{N_2} \, + \, C
\frac{\sqrt{n}}{R} \cdot \vphi \left(\frac{c \sqrt{n }}{R} \cdot \frac{c \sqrt{n}}{R_1} t  \right) \\ & \leq \frac{\tilde{C}}{\log^2 n} \, + \, C
\frac{\sqrt{n}}{R} \cdot \vphi \left(\frac{c \sqrt{n}}{{R_2}} t  \right). \end{align*}
We may combine the two inequalities above to obtain that for all $\xi \in S^{n-1}$ and $t \in \RR$,
\begin{equation*}
 \frac{1}{N_1 N_2} \sum_{i=1}^{N_1} \sum_{j=1}^{N_2} \vphi \left( t + R_1 \xi \cdot \theta_i  + R_2 \xi \cdot \eta_j \right) \leq \frac{\tilde{C}}{\sqrt{n} \log^{3/2} n} \, + \, C
\frac{\sqrt{n}}{R_2} \cdot \vphi \left(\frac{c \sqrt{n}}{{R_2}} t  \right). \tag*{\qedhere} \end{equation*}
\end{proof}

The reason we iterated Proposition \ref{prop_612} only twice and not thrice or more in the proof of Lemma \ref{lem_955}
is basically the non-optimal use  of the $\delta$-net alluded to in Remark \ref{rem_1131}.

\smallbreak Write $e_1,\ldots,e_n$ for the standard unit vectors in $\RR^n$. Let $\theta_1,\ldots,\theta_{N_1}$
and $\eta_1,\ldots,\eta_{N_2}$ be the unit vectors whose existence is proven in Lemma \ref{lem_955}.
We now define the centrally-symmetric convex body $K \subseteq \RR^n$ to be the convex hull of the $2(N_1 + N_2 + n)$ vectors
$$ \pm R_1 \theta_1,\ldots, \pm R_1 \theta_{N_1}, \quad \pm R_2 \eta_1,\ldots, \pm R_2 \eta_{N_2}, \quad \pm n e_1, \ldots, \pm n e_n. $$
We write $\gamma_n$ for the standard Gaussian measure in $\RR^n$, whose density is $$ x \mapsto (2 \pi)^{-n/2} \exp(-|x|^2/2). $$
We will use below the standard bound $(2 \pi)^{-1/2} \int_t^{\infty} \vphi(s) ds \leq \vphi(t)/2$ for all $t \geq 0$.

\begin{lemma} We have that $\displaystyle |K| \leq C^n$, where $C > 0$ is a universal constant. \label{cor_gluskin}
\end{lemma}

\begin{proof} We will mimick an argument by Gluskin \cite{G}. For a centrally-symmetric convex body $K \subseteq \RR^n$ we denote its polar body by
$$ K^{\circ} = \left \{ x \in \RR^n \, ; \, \forall y \in K, \, x \cdot y \leq 1  \right \}.$$
Let $Z$ be a standard Gaussian random vector in $\RR^n$. According to the Khatri-Sidak lemma  (\cite{Kh}, \cite{Si}, see also \cite{G}
for a simple proof),
\begin{align*}
\PP( Z \in 5 n K^{\circ} ) & = \PP \left( \forall i,j,k, \quad |R_1 Z \cdot \theta_i| \leq 5 n, \ |R_2 Z \cdot \eta_j| \leq 5 n \ \text{and} \ |n Z \cdot e_k| \leq 5  \right) \\ & \geq \prod_{i=1}^{N_1} \PP \left( |R_1 Z \cdot \theta_i| \leq 5 n \right) \cdot
\prod_{j=1}^{N_2} \PP \left( |R_2 Z \cdot \eta_j| \leq 5 n \right) \cdot \prod_{k=1}^{n} \PP \left( |n Z \cdot e_k| \leq 5 n \right).
\end{align*}
Since $Z \cdot \theta_i$ is a standard Gaussian random variable, we know that
$$ \PP \left( |R_1 Z \cdot \theta_i| \leq 5 n \right) = 1 - \frac{2}{\sqrt{2 \pi}} \int_{5n/R_1}^{\infty} \vphi(s) ds \geq 1 - \vphi(5 n / R_1). $$
Consequently,
$$  \gamma_n( 5 n K^{\circ} )  = \PP( Z \in 5 n K^{\circ} ) \geq \left( 1 - \vphi(5 n / R_1) \right)^{N_1} \cdot \left( 1 - \vphi(5 n / R_2) \right)^{N_2}
\cdot \left(1 - \vphi(5) \right)^n. $$
Recalling from (\ref{eq_130}) the values of our parameters, we obtain
$$ \gamma_n( 5 n K^{\circ} ) \geq \left( 1 - \frac{1}{n^{10}} \right)^{n^3} \cdot  \left( 1 - \frac{1}{(\log n)^{10}} \right)^{1 + n \log^3 n} \cdot c^n
\geq e^{-\tilde{c} n}. $$
Since the density of the measure $\gamma_n$ does not exceed $(2 \pi)^{-n/2}$, we conclude that
$$ |5 n K^{\circ} | > c^n. $$
The conclusion of the lemma now follows from the Santalo inequality (see e.g. \cite[Theorem 1.3.4]{BGVV}).
\end{proof}

Since $e_1,\ldots, e_n \in K / n$ while $K \subseteq \RR^n$ is convex and centrally-symmetric, we know that
\begin{equation}
K \supseteq \sqrt{n} B^n.  \label{eq_625}
\end{equation}
Recall that  $\int_{\RR^n} |x|^2 d \gamma_n(x) = n$. By the Markov-Chebyshev inequality,
\begin{equation}
\gamma_n \left( 2 \sqrt{n} B^n \right) \geq 1/2.\label{eq_313}
\end{equation}
Let us now define the probability measure $\mu$ to be the convolution
$$ \mu = \gamma_n * \left( \frac{1}{N_1 N_2 } \sum_{i=1}^{N_1} \sum_{j=1}^{N_2} \frac{\delta_{R_1 \theta_i + R_2 \eta_j} + \delta_{-R_1 \theta_i - R_2 \eta_j} }{2} \right)$$
where $\delta_y$ is the delta measure
at the point $y \in \RR^n$. Then $\mu$ is a probability measure in $\RR^n$.

\begin{lemma} $\displaystyle \mu(4 K) \geq 1/2$. \label{lem_647}
\end{lemma}

\begin{proof} For any $i$ and $j$, the convex body $4  K = 2 K + 2 K $ contains the set $\pm R_1 \theta_i \pm R_2 \eta_j + 2 \sqrt{n} B^n$,
according to (\ref{eq_625}). The measure $\mu$ is an average of translates of $\gamma_n$, each centered at a point
of the form $\pm R_1 \theta_i \pm R_2 \eta_j$. Consequently,
\begin{equation*} \mu (4 K) \geq  \gamma_n(2 \sqrt{n} B^n) \geq \frac{1}{2}. \tag*{\qedhere} \end{equation*}
\end{proof}

Write $g$ for the continuous density of the measure $\mu$. Setting $\vphi_n(x) =  \exp(-|x|^2/2)$
for $x \in \RR^n$, we have
$$ g(x) = \frac{1}{N_1 N_2} \sum_{i=1}^{N_1} \sum_{j=1}^{N_2} \frac{\vphi_n(x + R_1 \theta_i + R_2 \eta_j) +
\vphi_n(x - R_1 \theta_i - R_2 \eta_j)}{2 \cdot (2 \pi)^{n/2}}. $$
Note that for any $\xi \in S^{n-1}, z \in \RR^n$ and $t \geq 0$,
\begin{equation}
 \int_{t \xi + \xi^{\perp}} \frac{\vphi_n(x + z)}{(2 \pi)^{n/2}} dx = \frac{\vphi(t + z \cdot \xi)}{\sqrt{2 \pi}}. \label{eq_1600}
 \end{equation}

\begin{lemma} For any $\xi \in S^{n-1}$ and $t \geq 0$,
$$ \int_{t \xi + \xi^{\perp}} g \leq C \frac{\sqrt{\log \log n}}{\sqrt{n}}. $$
 \label{lem_637}
\end{lemma}

\begin{proof} By (\ref{eq_1600}) we have that
$$ \int_{t \xi + \xi^{\perp}} g
= \frac{1}{N_1 N_2} \sum_{i=1}^{N_1} \sum_{j=2}^{N_2}
\frac{\vphi \left( t + R_1 \xi \cdot \theta_i  + R_2 \xi \cdot \eta_j \right) + \vphi \left( -t + R_1 \xi \cdot \theta_i  + R_2 \xi \cdot \eta_j \right) }{2 \cdot \sqrt{2 \pi}}.  $$
Therefore, according to  Lemma \ref{lem_955} and as $\vphi \leq 1$,
$$ \int_{t \xi + \xi^{\perp}} g \leq \frac{C}{\sqrt{n} \log n} + C \frac{\sqrt{n}}{R_2} \leq \tilde{C} \frac{\sqrt{\log \log n}}{\sqrt{n}}, $$
completing the proof.
\end{proof}

\begin{proof}[{\bf Proof of Theorem \ref{thm_639}}] We set
$$ f(x) = g(x) \cdot 1_{4 K}(x) / \mu(4 K) $$
where $1_{4 K}$ is the function that equals one on $4 K$ and vanishes otherwise. Then $f$ is an even, continuous probability
density supported on $T = 4 K$. According  to Lemma \ref{lem_647},
$$ f(x) \leq 2 g(x) \qquad \text{for} \ x \in \RR^n. $$
Therefore, from Lemma \ref{lem_637}, for any $\xi \in S^{n-1}$ and $t \geq 0$,
\begin{equation}  \int_{T \cap (\xi^{\perp} + t \xi)} f(y) dy \leq 2 \int_{\xi^{\perp} + t \xi} g(y) dy \leq \frac{C \sqrt{\log \log n}}{\sqrt{n}}. \label{eq_648}
\end{equation}
We also know that $|T|^{1/n} \leq C$, according to Lemma \ref{cor_gluskin}. Therefore
the desired estimate (\ref{eq_649}) follows from (\ref{eq_648}).
\end{proof}

The left-hand side inequality in Theorem \ref{thm_530} clearly follows  from Theorem \ref{thm_639}.
Note also that Theorem \ref{thm_639} entails the optimality, up to a factor of $\log \log n$, of Corollary 1 from \cite{KZ}.

\section{Measures admitting tail bounds}

This section is devoted to the proof of Theorem \ref{thm_1026}, which is a modification of the proof of Theorem \ref{thm_639}.
We are given a dimension $n$ and $\alpha \in (0, 2]$. We may assume that $n \geq 10$
as otherwise the conclusion of the theorem is trivial.

\begin{lemma} Let $\Theta \in S^{n-1}$ be a random vector, distributed uniformly over $S^{n-1}$. Let $\xi \in S^{n-1}$ be a fixed unit vector.
Then,
$$ \PP \left( |\langle \Theta, \xi \rangle| \geq \frac{1}{\sqrt{n}} \right) \geq c \qquad \textrm{and} \qquad
\EE e^{ \left( \frac{\sqrt{n} |\langle \Theta, \xi \rangle|  + 1}{2} \right)^{\alpha} } \leq C_1, $$
where $c, C_1 > 0$ are universal constants.
 \label{lem_1705}
\end{lemma}

\begin{proof} Denote $Z = \langle \Theta, \xi \rangle$.
As in the proof of Lemma \ref{lem_540}, the density of $Z$
in the interval $[-1,1]$ is proportional to $\beta_n (1-t^2)^{(n-3)/2}$,
where $\beta_n$ satisfies $c \sqrt{n} \leq \beta_n \leq C \sqrt{n}$. We need
to prove that
\begin{equation}  \sqrt{n} \int_{-1}^1 1_{ \{ |s| \geq 1/\sqrt{n} \}} \left( 1 - s^2 \right)^{\frac{n-3}{2}} ds \geq c
\quad \textrm{and} \quad \sqrt{n} \int_{-1}^1 e^{ \left( \frac{\sqrt{n} |s| + 1}{2} \right)^{\alpha} } \left( 1 - s^2 \right)^{\frac{n-3}{2}} ds \leq C_1 \label{eq_1014}
\end{equation}
The left-hand side inequality in (\ref{eq_1014}) follows from
$$ \int_{-1}^1 1_{ \{ |s| \geq 1/\sqrt{n} \}} \left( 1 - s^2 \right)^{\frac{n-3}{2}} ds \geq 2 \int_{1/\sqrt{n}}^{2/\sqrt{n}} \left( 1 - s^2 \right)^{n/2} ds \geq \frac{2}{\sqrt{n}} \cdot (1 - 4/n)^{n/2} \geq \frac{c}{\sqrt{n}}. $$
As for the right-hand side inequality in (\ref{eq_1014}), we argue as follows:
\begin{align*}
\sqrt{n} & \int_{-1}^1 e^{ \left( \frac{\sqrt{n} |s| + 1}{2} \right)^{\alpha} } \left( 1 - s^2 \right)^{\frac{n-3}{2}} ds \leq \sqrt{n} \int_{-\infty}^{\infty}
e^{ \left( \frac{\sqrt{n} |s| + 1}{2} \right)^{\alpha} } e^{- s^2 n /3} ds = \int_{-\infty}^{\infty} e^{ \left( \frac{|t| + 1}{2} \right)^{\alpha} } e^{- t^2 /3} dt
\\ & \leq 14 e^{4^\alpha} + 2 \int_7^{\infty} e^{ \left( \frac{t + 1}{2} \right)^{\alpha} } e^{- t^2 /3} dt \leq
14  e^{16} + 2 \int_0^{\infty} e^{ \left( \frac{4t}{7} \right)^{2} } e^{- t^2 /3} dt = 14 \cdot e^{16} + \sqrt{147  \pi}.
\tag*{\qedhere}
\end{align*}
 \end{proof}

Let us now  introduce the parameter $$ N = \lceil n^8 \cdot \exp( 8 n^{\alpha/2} ) ) \rceil. $$
Let $\Theta_1,\ldots, \Theta_N$ be  independent random vectors, distributed uniformly in $S^{n-1}$.

\begin{lemma} Assume that $n \geq \bar{C}$.
Then with probability of at least $1 - e^{-n^2}$ of selecting $\Theta_1,\ldots,\Theta_N$ the following holds: For all $\xi \in S^{n-1}$,
$$ \frac{1}{N} \sum_{i=1}^N |\langle \Theta_i, \xi \rangle| \geq \frac{c}{\sqrt{n}} \qquad \textrm{and} \qquad
\frac{1}{N} \sum_{i=1}^N e^{ \left( \sqrt{n} |\langle \Theta_i, \xi \rangle| / 2 \right)^{\alpha} } \leq C. $$
Here, $c, C, \bar{C} > 0$ are universal constants. \label{lem_403}
  \end{lemma}

\begin{proof} The constant $\bar{C} > 10$ will be a sufficiently large universal constant whose value will be determined later on.
Fix $\xi \in S^{n-1}$ and set $$ Y_i = 1_{ \{ |\langle \Theta_i, \xi \rangle| \geq 1/ \sqrt{n} \} } \qquad \qquad (i=1,\ldots,N). $$
Then $Y_1,\ldots,Y_N$ are independent, identically-distributed random variables attaining values in $\{0, 1\}$, with $\PP(Y_i = 1) \geq c$
according to Lemma \ref{lem_1705}. By a standard estimate for the binomial distribution (see, e.g., \cite[Chapter 2]{BLM}),
$$
\PP \left( \frac{1}{N} \sum_{i=1}^N Y_i \geq c / 2 \right) \geq 1 - C_2 e^{-c_1 N}.
$$
Consequently, for any fixed $\xi \in S^{n-1}$,
\begin{equation}
\PP \left( \frac{1}{N} \sum_{i=1}^N |\langle \Theta_i, \xi \rangle| \geq \frac{c_2}{\sqrt{n}} \right) \geq
\PP \left( \frac{1}{N} \sum_{i=1}^N Y_i \geq c / 2 \right) \geq 1 - C_2 e^{-c_1 N}.
\label{eq_1159}
\end{equation}
Next, set $$ Z_i = \exp \left( \left( \frac{\sqrt{n} |\langle \Theta_i, \xi \rangle| + 1}{2} \right)^{\alpha} \right) \qquad \qquad (i=1,\ldots,N). $$
Then $1 \leq Z_i \leq \exp(n^{\alpha/2}) \leq \sqrt{N}$ while $\EE Z_i \leq C_3$ according to Lemma \ref{lem_1705}.
We may thus use Lemma \ref{lem_512} and conclude that for any fixed $\xi \in S^{n-1}$,
\begin{equation}  \PP \left( \frac{1}{N} \sum_{i=1}^N Z_i \leq 3 C_3 \right) = \PP \left( \frac{1}{N} \sum_{i=1}^N \frac{Z_i}{\sqrt{N}}  \leq \frac{3 C_3}{\sqrt{N}}  \right) \geq 1 - e^{-N C_3 / \sqrt{N}} = 1 - e^{-C_3 \sqrt{N}}.
\label{eq_1211} \end{equation}
We now select the universal constant $\bar{C} > 10$ large enough so that the assumption $n \geq \bar{C}$ implies
\begin{equation}   C_2 e^{-c_1 n^8} + e^{-C_3 n^4} \leq e^{-50 n^2} \qquad \textrm{and} \qquad \frac{\sqrt{n}}{c_2} \leq \frac{n^2}{2} \label{eq_228} \end{equation}
where $c_1, c_2, C_2$ are the constants from (\ref{eq_1159}) while $C_3$ is the constant from (\ref{eq_1211}).
Consider the functions
$$ f(t) = t \qquad \textrm{and} \qquad  g(t) = \exp \left \{ \left( \frac{\sqrt{n} t + 1}{2} \right) ^{\alpha} \right \} \qquad \qquad \textrm{for} \ 0 \leq t \leq 1. $$
Clearly $f$ is a $1$-Lipschitz function, while for any $0 \leq t \leq 1$,
$$ |g^{\prime}(t)| = g(t) \cdot \frac{\alpha \sqrt{n}}{2} \cdot \left( \frac{\sqrt{n} t + 1}{2} \right) ^{\alpha-1} \leq
2\sqrt{n} \cdot \exp \left \{ 2 \left( \frac{\sqrt{n} t + 1}{2} \right) ^{\alpha} \right \} \leq n e^{2 n^{\alpha/2}} \leq N^{1/4},
$$
where we used the elementary inequality $x^{\alpha-1} e^{x^{\alpha}} \leq 2 e^{2 x^{\alpha}}$ with $x = (\sqrt{n} t + 1) / 2 \geq 1/2$.
Hence $g$ is an $N^{1/4}$-Lipschitz function on $S^{n-1}$. Set $\delta = N^{-1/3}$ and let $\cF \subseteq S^{n-1}$
be a $\delta$-net of cardinality
$$  \#(\cF) \leq \left( \frac{5}{\delta} \right)^n \leq e^{6 n \log (N^{1/3}) } = N^{2n} \leq (n e^n + 1)^{16 n} \leq e^{40 n^2}. $$
From (\ref{eq_1159}), (\ref{eq_1211}) and (\ref{eq_228}), with probability of at least $1 - e^{-10 n^2}$ of selecting $\Theta_1,\ldots,\Theta_N$,
for all $\xi \in \cF$,
\begin{equation} \frac{1}{N} \sum_{i=1}^N |\langle \Theta_i, \xi \rangle| \geq \frac{c_2}{\sqrt{n}} \qquad \textrm{and} \qquad
\frac{1}{N} \sum_{i=1}^N \exp \left \{ \left( \frac{\sqrt{n} |\langle \Theta_i, \xi \rangle| + 1}{2} \right)^{\alpha} \right \} \leq 3 C_3. \label{eq_236}
\end{equation}
Since $\cF$ is a $\delta$-net, from (\ref{eq_236}) and the Lipschitz properties of $f$ and $g$ we obtain that with probability of at least $1 - e^{-10 n^2}$,
for all $\xi \in S^{n-1}$,
\begin{equation}
\frac{1}{N} \sum_{i=1}^N |\langle \Theta_i, \xi \rangle| \geq \frac{c_2}{\sqrt{n}} - \delta \qquad \textrm{and} \qquad
\frac{1}{N} \sum_{i=1}^N \exp \left \{ \left( \frac{\sqrt{n} |\langle \Theta_i, \xi \rangle| + 1}{2} \right)^{\alpha} \right \} \leq 3 C_3 + N^{1/4} \delta.
\label{eq_239} \end{equation}
However, $\delta = N^{-1/3} \leq \frac{1}{n^2} \leq \frac{c_2}{2 \sqrt{n}}$ according to (\ref{eq_228}), while $N^{1/4} \delta \leq 1$.
Hence the conclusion of the lemma follows from (\ref{eq_239}).
\end{proof}

We define the centrally-symmetric convex body $K \subseteq \RR^n$ to be the convex hull of the $2N + 2n$ points
$$ \pm R \Theta_1,\ldots, \pm R \Theta_{N}, \pm n e_1,\ldots, \pm n e_n, $$
where
$$ R = n^{1 -  \alpha / 4}. $$
From now on in this paper, we write $c, C, C_1, \tilde{C}$ etc. for various positive constants that depend solely on
$\alpha \in (0,2]$.

\begin{lemma} Assume that $n \geq \bar{C}$. Then with probability one, $\displaystyle |K| \leq C^n$, where $C, \bar{C} > 0$ depend solely on $\alpha$. \label{cor_gluskin2}
\end{lemma}

\begin{proof} As in the proof of Corollary \ref{cor_gluskin2}, it suffices to show that $\gamma_n( 10 n K^{\circ} ) \geq c^n$.
Let $Z$ be a standard Gaussian random vector in $\RR^n$, independent of the $\Theta_i$'s. By the Khatri-Sidak lemma,
\begin{align*}  \gamma_n( 10 & n K^{\circ} )  = \PP \left( \forall i,j \quad  |\langle Z, \Theta_i \rangle| \leq 10n/ R \ \ \textrm{and} \ \  |\langle Z, e_j \rangle| \leq 10 \right) \\ & \geq \left(1 - \vphi(10n/R) \right)^N \cdot (1 - \vphi(10))^n
\geq \left(1 - e^{-50 n^{\alpha/2} } \right)^{n^8 \exp(8 n^{\alpha/2}) + 1} \cdot (1 - \exp(-c))^n.
\end{align*}
However, $n^8 \exp(8 n^{\alpha/2}) + 1 \leq C \cdot e^{10 n^{\alpha/2}}$ where $C > 0$ depends solely on $\alpha$.
Moreover, $e^{10 n^{\alpha/2}} \geq 2$ assuming that $n \geq \bar{C}$ for some $\bar{C} > 0$ depending on $\alpha$.
Hence,
$$ \gamma_n( 10 n K^{\circ} ) \geq \left(1 - e^{-50 n^{\alpha/2} } \right)^{ C \cdot e^{10 n^{\alpha/2}} } \cdot c^n \geq \tilde{c}^n, $$
where $\tilde{c} > 0$ depends only on $\alpha$.
\end{proof}

We define the probability measure $\mu$ to be the convolution $ \mu = \gamma_n * \nu $
where
$$ \nu = \frac{1}{N} \sum_{i=1}^{N} \frac{\delta_{R \Theta_i} + \delta_{-R \Theta_i }}{2}. $$

\begin{lemma} Assume that $n \geq \bar{C}$. Then with probability one, $\displaystyle \mu(3 K) \geq 1 - C \exp(-c n) \geq 1/2$.
\label{lem_1527}
\end{lemma}

\begin{proof} Note that $K \supseteq \sqrt{n} B^n$ as $\pm n e_1,\ldots, \pm n e_n \in K$.
Consequently, the convex body $3  K$ contains the set $\pm R \Theta_i + 2 \sqrt{n} B^n$
for any $i=1,\ldots,N$. As in the proof of Lemma \ref{lem_647}, the measure $\mu$ is a mixture of translates of $\gamma_n$, each centered at a point
of the form $\pm R \Theta_i$. Therefore $\mu(3 K)$ is at least $\gamma_n(2 \sqrt{n} B^n)$ which in turn is at least $1 - C \exp(-c n) \geq 1/2$
by a standard estimate.
\end{proof}

\begin{lemma} Assume that $n \geq \bar{C}$.
Then with probability of at least  $1 - e^{-n^2}$, the measure $\mu$ admits $\psi_{\alpha}$-tails with parameters $(C, c)$.
Here, $c, C, \bar{C} > 0$ depend only on $\alpha$. \label{lem_529}
\end{lemma}

\begin{proof} We will assume that the event described in Lemma \ref{lem_403} holds true, which
happens with probability of at least  $1 - e^{-n^2}$. We need to show that for any $\xi \in S^{n-1}$, setting $E_{\xi} = \int_{\RR^n} |\langle x, \xi \rangle| d \mu(x)$, we have
\begin{equation}
\mu \left( \left \{ x \in \RR^n \, ; \, |\langle \xi, x \rangle| \geq t E_{\xi} \right \} \right) \leq C \exp(-c t^{\alpha}). \qquad \qquad \textrm{for} \ t > 0. \label{eq_431}
\end{equation}
Since the event described in Lemma \ref{lem_403} holds true, we know that
for any $\xi \in S^{n-1}$,
\begin{align}  \nonumber E_{\xi} & = \int_{\RR^n}  |\langle x, \xi \rangle| d \mu(x)  = \frac{1}{N} \sum_{i=1}^N \frac{1}{\sqrt{2 \pi}} \int_{-\infty}^{\infty}
\frac{|t - \langle \xi, R \Theta_i \rangle| + |t + \langle \xi, R \Theta_i \rangle|}{2} e^{-t^2/2} dt \\ & \geq   \frac{1}{N} \sum_{i=1}^N |\langle \xi, R \Theta_i \rangle| \geq c \frac{R}{\sqrt{n}}. \nonumber  \end{align}
Therefore,
\begin{align} \nonumber \int_{\RR^n} \exp \left \{ \left( c_1 \frac{ |\langle \xi, x \rangle|}{E_{\xi}} \right)^{\alpha} \right \} d \nu(x) & =
\frac{1}{N} \sum_{i=1}^N \exp \left \{ \left( c_1 \frac{ |\langle \xi, R \Theta_i \rangle|}{E_{\xi}} \right)^{\alpha} \right \}  \\ & \leq
\frac{1}{N} \sum_{i=1}^N \exp \left \{ \left(  \frac{ \sqrt{n} |\langle \xi,  \Theta_i \rangle|}{2} \right)^{\alpha} \right \} \leq C. \label{eq_508}
\end{align}
A standard application of the Markov-Chebyshev inequality based on (\ref{eq_508}), which appears e.g. in \cite[Section 2.4]{BGVV}, shows that for any $\xi \in S^{n-1}$ and $t > 0$,
\begin{equation}  \nu \left( \left \{ x \in \RR^n \, ; \, |\langle \xi, x \rangle| \geq t E_{\xi} \right \} \right) \leq  C \exp(-c t^{\alpha}). \label{eq_508_}
\end{equation}
Since $E_{\xi} \geq c R / \sqrt{n} \geq c$, we know that for any $\xi \in S^{n-1}$ and $t > 0$,
\begin{equation}  \gamma_n \left( \left \{ x \in \RR^n \, ; \, |\langle \xi, x \rangle| \geq t E_{\xi} \right \} \right) \leq
\gamma_n \left( \left \{ x \in \RR^n \, ; \, |\langle \xi, x \rangle| \geq c t \right \} \right)
\leq \tilde{C} e^{-\tilde{c} t^2}  \leq \tilde{C} e^{-\hat{c} t^{\alpha}}. \label{eq_511}
\end{equation}
Since $\mu = \gamma_n * \nu$, we deduce (\ref{eq_431}) from (\ref{eq_508_}) and (\ref{eq_511}).
\end{proof}

Write $g$ for the continuous density of the measure $\mu$. Thus
$$ g(x) = \frac{1}{N} \sum_{i=1}^{N} \frac{\vphi_n(x + R \Theta_i ) +
\vphi_n(x - R \Theta_i)}{2 \cdot (2 \pi)^{n/2}} \qquad \qquad (x \in \RR^n). $$

\begin{lemma} Assume that $n \geq \bar{C}$. Then with probability of at least $1 - n^{-n}$ the following holds:
For any $\xi \in S^{n-1}$ and $t \geq 0$,
$$ \int_{t \xi + \xi^{\perp}} g \leq \frac{C}{n^{(2 - \alpha)/4}}, $$
where $C, \bar{C} > 0$ depend only on $\alpha$. \label{lem_1533}
\end{lemma}

\begin{proof} We may assume that $\bar{C} \geq 5$ is large enough so that the assumption
$n \geq \bar{C}$ implies that $10 n \log n \leq N \leq e^{10 n}$. We may thus apply
Proposition \ref{prop_612}. According to the conclusion of this proposition, with probability of at least $1 - n^{-n}$ the following holds: For all $\xi \in S^{n-1}$ and $t \in \RR$,
\begin{equation}\label{eq_1601}
 \frac{1}{N} \sum_{i=1}^N \vphi \left( t + R \xi \cdot \Theta_i   \right) \leq \frac{C n \log n}{\min \{N, n^3\}} \, + \, C
\frac{\sqrt{n}}{R} \leq \frac{\tilde{C}}{ n^{(2 - \alpha) /4} },
\end{equation}
where we recall that $\tilde{C} = \tilde{C}(\alpha)$ depends solely on $\alpha \in (0, 2]$. Consequently, for all $\xi \in S^{n-1}$ and $t \geq 0$
we may use (\ref{eq_1600}) and obtain
\begin{equation*}
 \int_{t \xi + \xi^{\perp}} g
= \frac{1}{N} \sum_{i=1}^{N}
\frac{\vphi \left( t + R \xi \cdot \Theta_i  \right) + \vphi \left( R \xi \cdot \Theta_i - t \right) }{2 \cdot \sqrt{2 \pi}} \leq \frac{C}{n^{(2 - \alpha)/4}}. \tag*{\qedhere} \end{equation*}
\end{proof}

It is well-known (see, e.g., \cite[Section 2.4]{BGVV}) that if the probability measure $\mu$ admits $\psi_{\alpha}$-tails with parameters $(\beta, \gamma)$,
then the following reverse H\"older inequality holds true: for any $\xi \in S^{n-1}$,
\begin{equation}
\sqrt{ \int_{\RR^n} |\langle x, \xi \rangle|^2 d \mu(x) } \leq C_{\beta, \gamma} \int_{\RR^n} |\langle x, \xi \rangle| d \mu(x),
\label{eq_605}
\end{equation}
where $C_{\beta, \gamma} > 0$ depends only on $\beta$ and $\gamma$.

\begin{proof}[Proof of Theorem \ref{thm_1026}] We may assume that $n \geq \bar{C}$,
as otherwise the conclusion of the theorem is trivial. We may fix $\Theta_1,\ldots, \Theta_N \in S^{n-1}$
such that the events described in Lemma \ref{cor_gluskin2}, Lemma \ref{lem_1527}, Lemma \ref{lem_529}
and Lemma \ref{lem_1533} hold true. Set
$$ f(x) = g(x) \cdot 1_{3 K}(x) / \mu(3 K) \qquad \qquad (x \in T = 3 K), $$ an even, continuous probability
density supported on $T$. Note that $f \leq 2 g$, according  to Lemma \ref{lem_1527}. From Lemma \ref{lem_1533}, for any $\xi \in S^{n-1}$ and $t \geq 0$,
$$  \int_{T \cap (\xi^{\perp} + t \xi)} f(y) dy \leq 2 \int_{\xi^{\perp} + t \xi} g(y) dy \leq \frac{C}{n^{(2 - \alpha)/4}}
\leq \frac{C}{n^{(2 - \alpha)/4}} \cdot |T|^{-1/n}, $$
where the last passage is the content of Lemma \ref{cor_gluskin2}. This completes the proof of (i).

By Lemma \ref{lem_529}, the probability measure $\mu$ admits $\psi_{\alpha}$-tails with parameters $(C, c)$.
Write $\eta$ for the probability measure whose density is $f$.
According to Lemma \ref{lem_1527} and the Cauchy-Schwartz inequality, for any $\xi \in S^{n-1}$,
\begin{align} \nonumber
 \int_{\RR^n} |\langle x, \xi \rangle| d \eta(x) & \geq \int_{\RR^n} |\langle x, \xi \rangle| d \mu(x) - \int_{\RR^n \setminus T} |\langle x, \xi \rangle| d \mu(x)
\\  \nonumber & \geq \int_{\RR^n} |\langle x, \xi \rangle| d \mu(x) - \sqrt{\mu(\RR^n \setminus T)} \sqrt{\int_{\RR^n \setminus T} |\langle x, \xi \rangle|^2 d \mu(x) }
\\ & \geq (1 - C e^{-c n}) \int_{\RR^n} |\langle x, \xi \rangle| d \mu(x),  \label{eq_606}
\end{align}
where we used (\ref{eq_605}) in the last passage. Since $\eta \leq 2 \mu$, it thus follows from (\ref{eq_606}) that
also $\eta$ has $\psi_{\alpha}$-tails with parameters $(\tilde{C}, \tilde{c})$. This completes the proof of (ii).
\end{proof}

\medskip \noindent
{\Large \bf Appendix}

In this appendix we indicate how Bourgain's argument may be modified in order
to prove (\ref{eq_1009}). We may normalize and assume that $\mu$ is a probability measure. The first observation is that
denoting $M = \sup_{H\subseteq \RR^n} \mu^+(K \cap H)$, we have that for all $\theta \in S^{n-1}$,
\begin{align} \label{eq_425} \int_{\RR^n} |\langle x, \theta \rangle|^2 d \mu(x)  &= \int_{-\infty}^{\infty} t^2 \rho_{\theta}(t) dt = 4 \int_0^{\infty} t \left( \int_t^{\infty} \rho_{\theta}(s) ds \right) dt \\ & \geq 4 \int_0^{1/(2M)} t \left( \frac{1}{2} - \int_0^t \rho_{\theta} \right) dt \geq 4 \int_0^{1/(2M)} t \left( \frac{1}{2} - t M \right) dt = \frac{1}{12 M^2},
\nonumber \end{align}
where $\rho_{\theta}(t) = \mu^+(K \cap H_{\theta, t})$ for $H_{\theta, t} = \{ x \in \RR^n \, ; \, \langle x, \theta \rangle = t \}$.
We will use inequality (\ref{eq_425}) in order to replace the hyperplane sections $\mu^+(K \cap H)$ in (\ref{eq_1009}) by second moments of the probability measure $\mu$.
Thus, in order to prove (\ref{eq_1009}), it suffices to show that
\begin{equation}  \inf_{\theta \in S^{n-1}}
\sqrt{\int_{\RR^n} |\langle x, \theta \rangle|^2 d \mu(x)} \leq C(\beta,\gamma) \cdot n^{(2 - \alpha)/4} \cdot \log n  \cdot |K|^{1/n}.
\label{eq_422} \end{equation}
The next step is to reduce matters to the case where $\mu$ is {\it isotropic}, in the sense that the integral
$\int_{\RR^n} |\langle x, \theta \rangle|^2 d \mu(x)$ does not depend on $\theta \in S^{n-1}$. Indeed,
there exists a volume-preserving linear transformation $T$ such that the push-forward $T_* \mu$ is isotropic
(see, e.g., \cite[Section 2.3]{BGVV}). The replacement of $\mu$ by $T_* \mu$ and of $K$ by $T(K)$ does not alter the right-hand side of (\ref{eq_422}), while it does not decrease the left-hand side. Hence from now on we assume that $L_\mu^2 :=
\int_{\RR^n} |\langle x, \theta \rangle|^2 d \mu(x)$ does not depend on $\theta \in S^{n-1}$. In particular,
$$ \int_{\RR^n} |x|^2 d \mu(x) = \sum_{i=1}^n \int_{\RR^n} |\langle x, e_i \rangle|^2 d \mu(x) = n L_{\mu}^2. $$
From the Markov-Chebyshev inequality it thus follows that $\mu( 2 \sqrt{n} L_{\mu} B^n ) \geq 1/2$. As in the proof of Proposition 3.3.3 in \cite{BGVV}, we may now use the $\psi_{\alpha}$-condition in order to conclude that
\begin{equation}  \int_{\RR^n} |\langle x, \theta \rangle|^2 d \mu(x) \leq \tilde{C} \int_{2 \sqrt{n} L_{\mu} B^n} |\langle x, \theta \rangle|^2 d \mu(x) \qquad \text{for all} \ \theta \in S^{n-1}, \label{eq_526}
\end{equation}
where $\tilde{C}$, like all constants in this Appendix, depends solely on $\beta$ and $\gamma$. The next step
is ``reduction to small diameter'', which means that in proving (\ref{eq_422}) we would like to reduce matters
to the case where $\mu$ is isotropic with
\begin{equation} \mu( C \sqrt{n} L_{\mu} B^n ) = 1, \label{eq_527} \end{equation}
for some constant $C$.
The argument for this reduction in \cite[Section 3.3.1]{BGVV}, which involves conditioning $\mu$ to the ball $2 \sqrt{n} L_{\mu} B^n$, applies almost verbatim thanks to (\ref{eq_526}). We may thus assume that (\ref{eq_527}) holds true,
or equivalently that $|\langle x, \theta \rangle| \leq C L_{\mu} \sqrt{n}$
for all $\theta \in S^{n-1}$ and all $x \in \RR^n$ in the support of the measure $\mu$. Therefore,
\begin{equation}  \int_{\RR^n} \exp \left \{ \left( \frac{\langle x, \theta \rangle}{\bar{C} L_{\mu} n^{(2 - \alpha)/4}} \right)^2 \right \}
d \mu(x) \leq \int_{\RR^n} \exp \left \{ \frac{C^{2-\alpha}}{\bar{C}^2} \cdot \frac{|\langle x, \theta \rangle|^{\alpha} }{L_{\mu}^\alpha }  \right \}
d \mu(x) \leq \tilde{C}, \label{eq_1154} \end{equation}
where the last inequality follows from the $\psi_{\alpha}$-condition and a suitable choice of the constant $\bar{C}$  (see \cite[Section 2.4]{BGVV} for standard
computations related to the $\psi_{\alpha}$-condition). Inequality (\ref{eq_1154}) implies  that  for any $\theta \in S^{n-1}$,
\begin{equation}
 \left \|  \langle \cdot, \theta \rangle \right \|_{\psi_2} \leq C \cdot L_{\mu} \cdot n^{(2 - \alpha) / 4}.
 \label{eq_536} \end{equation}
Once we proved the $\psi_2$-estimate in (\ref{eq_536}), we may proceed
as in the proof of Theorem 3.3.5 in \cite{BGVV}, and use Talagrand's comparison theorem, the $\ell$-position of Figiel and Tomczak-Jaegermann, and Pisier's estimate for the Rademacher projection. This establishes the desired inequality (\ref{eq_422}) in the isotropic case.

{
}

\bigskip

\noindent \textsc{Bo'az \ Klartag}: Department of Mathematics, Weizmann Institute of Science, Rehovot 76100 Israel, and
School of Mathematical Sciences, Tel Aviv University, Tel Aviv 69978.

\smallskip

\noindent \textit{E-mail:} \texttt{boaz.klartag@weizmann.ac.il}\bigskip

\noindent \textsc{Alexander \ Koldobsky}: Department of
Mathematics, University of Missouri, Columbia, MO 65211.

\smallskip

\noindent \textit{E-mail:} \texttt{koldobskiya@missouri.edu}


\begin{thebibliography}{99}
\addcontentsline{toc}{section}{References}
\setlength{\itemsep}{1pt}

\bibitem{BLM} Boucheron, S., Lugosi, G., Massart, P., {\it Concentration inequalities. A nonasymptotic theory of independence. }
With a foreword by Michel Ledoux. Oxford University Press, Oxford, 2013.

\bibitem{BGVV} Brazitikos, S., Giannopoulos, A., Valettas, P., Vritsiou, B.-H., {\it Geometry of isotropic convex bodies. }
Mathematical Surveys and Monographs, 196. American Mathematical Society, Providence, RI, 2014.

\bibitem{Bou1} J.~Bourgain, {\it On high-dimensional maximal functions associated to convex bodies}, Amer. J. Math., {108}, (1986), 1467–-1476.

\bibitem{Bou2} J.~Bourgain, {\it  Geometry of Banach spaces and harmonic analysis },
Proceedings of the International Congress of Mathematicians, Vol. 1, 2 (Berkeley, Calif., 1986), Amer. Math. Soc., Providence, RI, (1987), 871–-878.


\bibitem{Bourgain-1991} {J.~Bourgain}, {\it On the distribution of polynomials
on high-dimensional convex sets},
Geom. aspects of Funct. Anal. (GAFA seminar notes), Israel Seminar, Springer Lect.
Notes in Math. { 1469} (1991),
127--137.

\bibitem{CGL} {G.~Chasapis, A.~Giannopoulos and D-M.~Liakopoulos},
{\it Estimates for measures of lower dimensional sections of convex bodies}, Adv. Math. { 306} (2017), 880--904.

\bibitem{G} A.~Giannopoulos, {\it On some vector balancing problems,}
Studia Math. 122 (1997), 225--234.

\bibitem{Glu} E. D. Gluskin, {\it
Extremal properties of orthogonal parallelepipeds and their applications to the geometry of Banach spaces,}
Mat. Sb. (N.S.) 136(178) (1988), no. 1, 85--96; English translation in
Math. USSR-Sb. 64 (1989), no. 1, 85–-96.

\bibitem{GLW} P.~R.~Goodey, E.~Lutwak and W.~Weil, {\it Functional analytic characterization of classes
of convex bodies}, Math. Z. 222 (1996), 363--381.

\bibitem{Kh}
C. G. Khatri, {\it On certain inequalities for normal
distributions and their applications to simultaneous
confidence bounds,} Ann. Math. Statist. 38 (1967), 1853--1867.

\bibitem{Klartag-2006}  {B.~Klartag}, {\it On convex perturbations with a bounded
isotropic constant}, Geom. Funct. Anal. (GAFA) { 16}  (2006), 1274--1290.

\bibitem{KM}  {B.~Klartag} and {E.~Milman}, {\it Centroid bodies and the logarithmic Laplace transform - a unified approach},
 J. Funct. Anal.  262 (2012), 10--34.

\bibitem{Koldobsky-1998} {A.~Koldobsky}, {\it Intersection bodies, positive definite
distributions and the Busemann-Petty problem}, Amer. J. Math. { 120} (1998),
827--840.

\bibitem{Koldobsky-book} A.~Koldobsky, {\it Fourier analysis in convex geometry},
Amer. Math. Soc., Providence RI, 2005.

\bibitem{Koldobsky-2012} {A.~Koldobsky}, {\it A hyperplane inequality for measures of convex bodies in $\RR^n, n\le 4$},
Discrete Comput. Geom. { 47} (2012), 538--547.

\bibitem{Koldobsky-2014} {A.~Koldobsky}, {\it A $\sqrt{n}$ estimate for measures of hyperplane sections of convex bodies,} Adv. Math. { 254} (2014), 33--40.

\bibitem{Koldobsky-2015} A.~Koldobsky, {\it Slicing inequalities for measures of convex bodies,}
Adv. Math. 283 (2015), 473--488.

\bibitem{KP} A.~Koldobsky and A.~Pajor, {\it A remark on measures of sections of $L_p$-balls},
Geom. aspects of Funct. Anal. (GAFA seminar notes), Israel Seminar, Springer Lect. Notes in Math. 2169 (2017), 213--220.


\bibitem{KZ} A.~Koldobsky and A.~Zvavitch, {\it An isomorphic version of the Busemann-Petty problem for arbitrary measures}, Geom. Dedicata 174 (2015), 261–-277.

\bibitem{Lo} G.~Lozanovskii, {\it On some Banach lattices,} Siberian Math. J. 10 (1969), 419--
431.

\bibitem{Lutwak} { E.~Lutwak}, {\it Intersection bodies and dual mixed volumes},
Adv. Math. {71} (1988), 232--261.

\bibitem{pisier} G.~Pisier, {\it The volume of convex bodies and Banach space geometry}, Cambridge University Press, Cambridge, 1989.

\bibitem{Si}  Z.~Sidak, {\it On multivariate normal probabilities of rectangles: their dependence on
correlations}, Ann. Math. Stat. 39 (1968), 1425-1434.
\end{thebibliography}
\end{document}